\documentclass[11pt,reqno]{amsart}
\usepackage[colorlinks=true,linkcolor=blue,citecolor=magenta]{hyperref}
\usepackage{graphics, bbm, bigints, relsize, fullpage, bbm}
\usepackage{amssymb,amsmath,amsthm, mathrsfs}
\usepackage{color}
\usepackage{tikz}
\usepackage[mathscr]{euscript}
\usepackage{enumerate}
\usepackage[conditional, light, first, bottomafter]{draftcopy}
\draftcopyName{DRAFT\space\today}{130}
\draftcopySetScale{65}

\usepackage{lmodern}
\usepackage{microtype}
\makeatletter
\renewcommand{\section}{\@startsection%
{section}%
{1}%
{0em}%
{1.7em}%
{1.2em}%
{\normalfont\large\centering\bfseries}}
\renewcommand{\@seccntformat}[1]%
{\csname the#1\endcsname.\hspace{0.5em}}
\makeatother
\renewcommand{\thesection}{\arabic{section}}

\numberwithin{equation}{section}
\renewcommand\appendix{\par
\setcounter{section}{0}%
\setcounter{subsection}{0}%
\setcounter{theorem}{0}
\setcounter{table}{0}
\setcounter{figure}{0}
\gdef\thetable{\Alph{table}}
\gdef\thefigure{\Alph{figure}}
\section*{Appendix}
\gdef\thesection{\Alph{section}}
\setcounter{section}{1}}
\newtheorem{theorem}{Theorem}[section]
\newtheorem{proposition}[theorem]{Proposition}

\newtheorem{corollary}[theorem]{Corollary}
\theoremstyle{definition}
\newtheorem{definition}{Definition}
\newtheorem{remark}{Remark}

\newcommand{\ie}{\emph{i.e.\kern.2em}}
\newcommand{\cf}{\emph{cf.\kern.2em}}
\newcommand{\viz}{\emph{viz.\kern.2em}}
\newcommand{\eg}{\emph{e.g.\kern.2em}}

\newcommand{\complex}{\mathbb{C}}

\newcommand{\inner}[2]{\left\langle#1,#2\right\rangle}

\newcommand{\cH}{{\mathcal H}}
\newcommand{\cc}[1]{\overline{#1}}

\DeclareMathOperator{\dom}{dom}

\DeclareMathOperator{\ran}{ran}

\begin{document}

\title[Asymptotics of continuous media with high-contrast inclusions]{Operator-norm resolvent asymptotic analysis of continuous media with high-contrast inclusions}
\author{Kirill D. Cherednichenko}
\address{Department of Mathematical Sciences, University of Bath, Claverton Down, Bath, BA2 7AY, United Kingdom}
\email{cherednichenkokd@gmail.com}
\author{Alexander V. Kiselev}
\address{Faculty of Mathematics and Computer Science, St.Petersburg State University, 29 14th line of V.O., St.Petersburg, 199178 Russia {\sc and} ITMO University, 49A Kronverkskii pr., St.Petersburg, 197101 Russia}
\email{alexander.v.kiselev@gmail.com}
\author{Luis O. Silva}
\address{Departamento de F\'{i}sica Matem\'{a}tica, Instituto de Investigaciones en Matem\'aticas Aplicadas y en Sistemas, Universidad Nacional Aut\'onoma de M\'exico, C.P. 04510, M\'exico D.F. {\sc and} Department of Mathematical Sciences, University of Bath, Claverton Down, Bath, BA2 7AY, United Kingdom}
\email{silva@iimas.unam.mx}



\begin{abstract}
Using a generalisation of the classical notion of the Weyl $m$-function and the related formulae for the resolvents of boundary-value problems, we analyse the asymptotic behaviour of solutions to a ``transmission problem" for a high-contrast inclusion in a continuous medium, for which we prove the operator-norm resolvent convergence to a limit problem of ``electrostatic" type. In particular, our results imply the convergence of the spectra of high-contrast problems to the spectrum of the limit operator, with order-sharp convergence estimates. The approach developed in the paper is of a general nature and can thus be successfully applied in the study of other problems of the same type.
\end{abstract}

\maketitle

\section{Introduction}
\label{sec:introduction}

Parameter-dependent problems for differential equations have traditionally attracted much interest within applied mathematics, by virtue of their potential for
replacing complicated formulations with more straightforward, and often explicitly solvable, ones. This drive has led to a plethora of asymptotic techniques, from perturbation theory to multi-scale analysis, covering a variety of applications to physics, engineering, and materials science. It would be an insurmountable task to give a comprehensive review of the related literature. Notwithstanding the classical status of this subject area, problems that require new ideas continue emerging, often motivated by novel wave phenomena. One of the recent application areas of this kind is provided by composites and structures involving components with highly contrasting material properties  (stiffness, density, refractive index). Mathematically, such problems lead to boundary-value formulations for classical operators (such as the Laplace operator), but with parameter-dependent coefficients. For example, problems of this kind have arisen in the study of periodic composite media with high contrast (or ``large coupling")
between the material properties of the components, see \cite{HempelLienau_2000}, \cite{Zhikov2000}, \cite{CherErKis}.

In the present work, we consider a prototype large-coupling transmission problem, posed on a bounded domain $\Omega\subset{\mathbb R}^d,$ $d=2,3,$ see Fig.\,\ref{fig:kaplya}, containing a ``low-index" (equivalently, ``high propagation speed") inclusion $\Omega_-,$ located at a positive distance to the boundary $\partial\Omega.$ Mathematically, this is modelled by a ``weighted" Laplacian $-a_{\pm}\Delta$, where $a_+=1$ (the weight on the domain
$\Omega_+:=\Omega\setminus\overline{\Omega}_-$), and $a_-\equiv a$ (the weight on the domain $\Omega_-$) is assumed to be large, $a_-\gg1.$ This is supplemented by the Neumann boundary condition
$\partial u/\partial n=0$ on the outer boundary $\partial\Omega,$
where $n$ is the exterior normal to $\partial\Omega,$ and ``natural" continuity conditions on the ``interface" $\Gamma:=\partial\Omega_-.$ For each $a,$ we consider time-harmonic vibrations
of the physical domain
represented by $\Omega,$ described by the eigenvalue problem for an appropriate operator in $L^2(\Omega).$

\begin{figure}[h!]
\begin{center}
\includegraphics[scale=0.7]{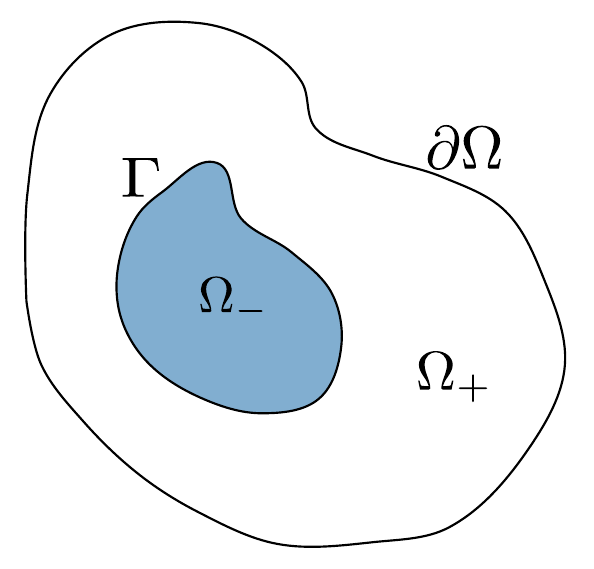}
\end{center}
\caption{Domain with a ``stiff" inclusion.\label{fig:kaplya}}
\end{figure}

A formal asymptotic argument using expansions in powers of  $a^{-1}$ suggests that convergent eigenfunction sequences for the above eigenvalue problems should converge (as $a\to\infty$) to either a constant or a function of the form
\[
v-\frac{1}{|\Omega|}\int_{\Omega_+}v,
\]
where
 $v$ satisfies the spectral boundary-value problem (BVP)
\begin{equation}
\label{electrostatic}
-\Delta v=z\biggl(v-\frac{1}{|\Omega|}\int_{\Omega_+}v\biggr)\ \ \ {\rm in}\ \Omega_+,\ \quad v\vert_\Gamma=0,
\qquad \dfrac{\partial v}{\partial n}\biggr\vert_{\partial\Omega}=0.
\end{equation}
Here the spectral parameter $z$ represents
 the ratio of the size of the original physical domain to the wavelength in its part
 represented by $\Omega_+.$

The problem (\ref{electrostatic})
 is related to the so-called ``electrostatic problem'' discussed in \cite[Lemma 3.4]{Zhikov_2004}, see also \cite{AKKL} and references therein, namely the eigenvalue problem for the self-adjoint operator $Q$ defined by the quadratic form
 \begin{equation}
 q(u, u)=\int_{\Omega_+}\nabla v\cdot\overline{\nabla v},\qquad u=v+c,\quad v\in H^1_{0,\Gamma}:=\bigl\{v\in H^1(\Omega_+),\ v\vert_\Gamma=0\bigr\},\quad c\in{\mathbb C}
 \label{qform}
 \end{equation}
on the Hilbert space $L^2(\Omega_+)+\mathbb{C}$, treated as a subspace of $L^2(\Omega)$.

Indeed (see \cite{Zhikov_2004} for details), it is easily seen that the eigenvalue problem $Qu=z u$ is solvable either when $z=0$, in which case $u=c$, or for $z>0$ such that the problem \eqref{electrostatic} admits a non-trivial solution. Thus, the formal asymptotic argument suggests that the limiting spectrum is precisely that of the electrostatic problem.

Denote by $A_0^+$ the Laplacian $-\Delta$ on $\Omega_+,$ subject to the Dirichlet condition on $\Gamma$ and the Neumann boundary condition
on $\partial\Omega$ and write the function $v$ in (\ref{electrostatic}) in the form of an eigenfunction series
\[
v=\sum_{j=1}^\infty d_j\phi_j^+,
\]
where $\lambda^+_j,$ $\phi^+_j,$ $j=1,2,\dots,$ are the eigenvalues and the corresponding orthonormal eigenfunctions, respectively, of $A_0^+.$
Noticing that the function ${\mathbbm 1}_+(x)=1,$ $x\in\Omega_+,$ can be written as
\[
{\mathbbm 1}_+=\sum_{j=1}^\infty\biggl(\int_{\Omega_+}\overline{\phi^+_j}\biggr)\phi^+_j,
\]
we obtain
\[
\sum_{j=1}^\infty d_j\lambda^+_j\phi^+_j=\sum_{j=1}^\infty z\biggl(d_j-\frac{1}{|\Omega|}\int_{\Omega_+}v\int_{\Omega_+}\overline{\phi^+_j}\biggr)\phi^+_j,
\]
and therefore
\begin{equation}
v=-\frac{z}{|\Omega|}\biggl(\int_{\Omega_+}v\biggr)\sum_{j=1}^\infty(\lambda_j^+-z)^{-1}\biggl(\int_{\Omega_+}\overline{\phi^+_j}\biggr)\phi^+_j,
\label{pre_final}
\end{equation}
as long as $z\neq \lambda^+_j,$ $j=1,2,\dots.$ Taking the integral over $\Omega_+$ on both sides of (\ref{pre_final}) and assuming that the integral of $v$ over $\Omega_+$ does not vanish yields, by incorporation of $z=0$  into the answer,
\begin{equation}
z\Biggl[|\Omega|+z\sum_{j=1}^\infty(\lambda_j^+-z)^{-1}\biggl|\int_{\Omega_+}\phi^+_j\biggr|^2\Biggr]=0.
\label{electrostatic_spectrum}
\end{equation}

Thus, the spectrum of the electrostatic problem is the union of two sets: a) the set of $z$ solving the equation \eqref{electrostatic_spectrum}
  and b) the set of those eigenvalues $\lambda^+_j$
for which the corresponding eigenfunction $\phi^+_j$ has zero mean over $\Omega_+.$

The main result of the present paper, which is the norm-resolvent asymptotics for the operator of the BVP introduced above, yields in particular the description \eqref{electrostatic_spectrum} for the limiting spectrum of the problem, together with an order-sharp estimate on the rate of the convergence, as $a\to +\infty$.

Relations similar to (\ref{electrostatic_spectrum}) appear in the analysis of periodic problems with micro-resonances (``metamaterials'') \cite{CherErKis}, where they provide zeros of the functions describing the dispersion of waves propagating through media modelled by such problems.


The present paper is a development of the recent study \cite{KCher,KCherYulia,KCherYuliaNab,CherErKis} aimed at implementing the
ideas of the boundary triples theory as proposed by Ryzhov \cite{Ryzh_spec} (in its turn, this analysis heavily draws upon the celebrated Birman-Kre\u\i n-Vi\v sik theory \cite{MR0080271, MR0024574, MR0024575, MR0051404}) in the context of problems of materials science and
wave propagation in inhomogeneous media.
Our recent papers cited above
have shown that the language of
boundary triples is particularly fitting for the analysis of composite media, as one of the key difficulties
in their analysis stems from the presence of interfaces (i.e., boundaries between individual material components) through which
an exchange of energy between different components of the medium takes
place. 
We point out that the papers \cite{ChKS_OTAA,CherednichenkoKiselevSilva,ChKS4} further demonstrate that an additional value of using the boundary triples approach is that the framework of functional models and the approach to scattering theory based thereupon (\cite{Drogobych,MR573902}) can be formulated in the most natural terms of Dirichlet-to-Neumann maps pertaining to the interfaces.


An asymptotic analysis of the static (or ``equilibrium") version of the above problem, where $z=0$ and a forcing term is added to the right-hand side, has been carried out in \cite{AKKL}, in the context of isotropic elasticity (which additionally implies that two material parameters are present, the so-called Lam\'{e} coefficients). The authors of \cite{AKKL}, using
the representation of solutions in terms of boundary layers, prove ``strong'' resolvent convergence of the original problem to the resolvent version of the ``electrostatic" problem (\ref{electrostatic}) (albeit framed in the context of linearised elasticity), still with $z=0.$
In the work \cite{Panasenko} different methods were used to obtain the spectral convergence; however, neither the effective operator of the ``limiting'' medium nor the norm-resolvent convergence to it were discussed.
 We argue that the approach we present here allows one to improve such results in two respects: a) the new estimates are of the operator-norm resolvent type, implying, in particular, the control of the convergence of the associated spectra and the exponential groups; b) our estimates are uniform with respect to the ``contrast" parameter $a$ and are order-sharp, i.e. the rate of convergence in terms of $a\to 0$ cannot be improved further.

We briefly outline the contents of the paper. In Section \ref{sec:triples-bvp}, we recall the main points of the abstract
construction of \cite{Ryzh_spec}  and introduce the key tools for our analysis. These include a representation for the resolvents of a class of boundary-value problems in terms of the $M$-operator. Using these  general formulae, in Section \ref{sec:example} we study the asymptotic behaviour of the operators corresponding to transmission problems for two-component media with contrasting material properties, as described above.
The asymptotic approximation of the spectra is discussed at the end of the paper, leading to the characterisation \eqref{electrostatic_spectrum}.

\section{Ryzhov triples for boundary-value problems}
\label{sec:triples-bvp}

In this section we follow \cite{Ryzh_spec} in outlining an operator framework suitable for
dealing with boundary-value problems.

\subsection{The boundary triple framework}
\label{BTF_sec}

The starting point of our construction is a
self-adjoint operator $A_0$ in a separable Hilbert space $\cH$ with
$0\in\rho(A_0)$, where $\rho(A_0),$ as usual, denotes the resolvent
set of $A_0$. Alongside  $\cH$, we consider an auxiliary Hilbert
space $\mathcal{E}$ and a bounded operator $\Pi: \mathcal{E}\to\mathcal{H}$
such that
\begin{equation*}
  \dom(A_0)\cap\ran(\Pi)=\{0\}\quad\text{and}\quad \ker(\Pi)=\{0\}.
\end{equation*}
Since $\Pi$ has a trivial kernel, there is a left inverse $\Pi^{-1},$ so that
$\Pi^{-1}\Pi=I_{\mathcal{E}}.$
We define
\begin{equation}
\label{eq:definition-A}
\begin{split}
  \dom(A)&:=\dom(A_0)\dotplus\ran(\Pi),\\
      A&:A_0^{-1}f+\Pi\phi\mapsto f,
\qquad f\in\mathcal{H},\phi\in\mathcal{E},
\end{split}
\end{equation}
\begin{equation}
\label{eq:definition-G-0}
\begin{split}
  \dom(\Gamma_0)&:=\dom(A_0)\dotplus\ran(\Pi),\\
      \Gamma_0&:A_0^{-1}f+\Pi\phi\mapsto \phi,
\qquad f\in\mathcal{H},\phi\in\mathcal{E},
\end{split}
\end{equation}
where neither $A$ nor $\Gamma_0$ is assumed closed or indeed closable.
The operator given in (\ref{eq:definition-A}) is the null extension of
$A_{0}$, while (\ref{eq:definition-G-0}) is the null extension of
$\Pi^{-1}$.
Note also that
\begin{equation*}
  \ker(\Gamma_0)=\dom(A_0)\,.
\end{equation*}
For $z\in\rho(A_0)$, consider the \emph{abstract} spectral BVP
\begin{equation}
\label{eq:first-abstract-bv-problem}
  \begin{cases}
    Au = zu,\\[0.2em]
    \Gamma_0 u = \phi,\qquad \phi\in\mathcal{E},
  \end{cases}
\end{equation}
where the second equation is seen as a boundary
condition.  As it is asserted in \cite[Thm.\,3.1]{Ryzh_spec}, there is
a unique solution $u$ of the BVP
(\ref{eq:first-abstract-bv-problem}) for any
$\phi\in\mathcal{E}$. Thus, there is an operator (clearly linear)
which assigns to any $\phi\in\mathcal{E}$ the solution $u$ of
(\ref{eq:first-abstract-bv-problem}). This operator is called the
solution operator for $A$ and is denoted by\footnote{The function $\gamma$ is
sometimes referred to as the $\gamma$-field.} $\gamma_z.$ An explicit
expression for it in terms of $A_0$ and $\Pi$ is obtained as
\begin{equation}
  \label{eq:solution-operator}
  \gamma_z:\phi\mapsto (I+z(A_0-zI)^{-1})\Pi\phi
\end{equation}
for any $z\in\rho(A_0)$.
Note that
\begin{equation*}
  I+z(A_0-zI)^{-1}=(I-zA_0^{-1})^{-1}
\end{equation*}
and that
(\ref{eq:definition-G-0}) and (\ref{eq:solution-operator}) immediately
imply
\begin{equation*}
  \Gamma_0\gamma_z=I_{\mathcal{E}}\,.
\end{equation*}
By (\ref{eq:solution-operator}) and a simple calculation,
one has
\begin{equation*}
  \ran\gamma_z=\ker(A-zI)\,.
\end{equation*}
We remark that, since $A$ is not required to be closed,
$\ran \gamma_z$ is not necessarily a subspace. This is precisely the kind of situation that commonly occurs in the analysis of BVPs.

In what follows, we consider (abstract) BVPs of the form
(\ref{eq:first-abstract-bv-problem}) associated with the operator $A$, with variable boundary conditions. To this
end, for a self-adjoint operator $\Lambda$ in $\mathcal{E},$ define
\begin{equation}
\label{eq:definition-G-kappa}
\begin{split}
  \dom(\Gamma_1)&:=\dom(A_0)\dotplus\Pi\dom(\Lambda),\\[0.3em]
  \Gamma_1&:A_0^{-1}f+\Pi\phi\mapsto\Pi^*f+\Lambda\phi,\qquad
  f\in\mathcal{H},\phi\in\dom(\Lambda).
\end{split}
\end{equation}
The operator $\Lambda$ can thus be seen as a parameter for the boundary operator $\Gamma_1.$

On the basis of \eqref{eq:solution-operator}, one obtains from
\eqref{eq:definition-G-kappa}  (see \cite[Eq.\,3.7]{Ryzh_spec}) that
\begin{equation*}
  \gamma_{\cc{z}}^*=\Gamma_1(A_0-zI)^{-1}\,.
\end{equation*}
Also, according to \cite[Thm.\,3.2]{Ryzh_spec}, the following Green's type identity holds:
\begin{equation*}
  \inner{Au}{v}_{\mathcal{H}}-\inner{u}{Av}_{\mathcal{H}}
=\inner{\Gamma_{1}u}{\Gamma_{0}v}_{\mathcal{E}}
-\inner{\Gamma_{0}u}{\Gamma_{1}v}_{\mathcal{E}},\qquad u,v\in\dom(\Gamma_{1})\,.
\end{equation*}


The above framework for triples $(A_{0},\Lambda,\Pi)$ stems from the Birman-Krein-Vi\v{s}ik theory
\cite{MR0080271,MR0024574,MR0024575,MR0051404}, rather than the theory
of boundary triples \cite{Gor}. We employ it next to introduce the notion of an $M$-operator, generalising the well-known notion of a Dirichlet-to-Neumann map in the context
of BVPs.
This generalisation helps us achieve two goals: on the one hand, it allows us to treat a transmission (rather than a boundary-value) problem, and on the other hand, it enters the formulae for operators resolvents that we will use for obtaining operator-norm error estimates in the large-coupling asymptotic regime.

\subsection{Definition and properties of the $M$-operator}
\label{M_sec}

Based on the notion of a triple, we now define the mentioned abstract version of the Dirichlet-to-Neumann map.
\begin{definition}
  \label{def:m-function}
  For a given triple $(A_{0},\Lambda,\Pi)$, define the operator-valued $M$-function associated with $A_0$ as follows. For any $z\in\rho(A_0)$, the operator $M(z)$ in $\mathcal{E}$  is defined by
  \begin{equation*}
    M(z): \phi\mapsto\Gamma_1\gamma_z\phi,\qquad \phi\in{\rm dom}\bigl(M(z)\bigr):=\dom(\Lambda).
  \end{equation*}
\end{definition}

A detailed description of how one casts in the language of boundary triples classical boundary-value problems, such as the Dirichlet problem for the Laplace operator on a bounded domain with sufficiently regular boundary, can be found in \cite{Ryzh_spec, ChKS4}.


Taking into account
(\ref{eq:definition-G-kappa}), one concludes from
Definition~\ref{def:m-function} that
\begin{equation}
\label{M_operator}
  M(z)=\Lambda + z\Pi^*(I-zA_0^{-1})^{-1}\Pi.
\end{equation}
Also, due to the self-adjointness of
$\Lambda$, one has
\begin{equation*}
  M^*(z)=M(\cc{z}).
\end{equation*}
Moreover, it is checked that  $M$ is an unbounded operator-valued Herglotz function, {\it i.e.,} $M(z)-M(0)$ is analytic and $\Im M(z)\ge 0$ whenever
$z\in\complex_+$. It is shown in \cite[Thm.\,3.3(4)]{Ryzh_spec} that
\begin{equation*}
  M(z)\Gamma_{0}u_{z}=\Gamma_{1}u_{z}\qquad\forall u_{z}\in\ker(A-zI)\cap\dom(\Gamma_{1}).
\end{equation*}
In this work we consider extensions (self-adjoint and non-selfadjoint) of the ``minimal'' operator
\begin{equation}
 \label{eq:B-definition}
  \widetilde{A}:=A_0\vert_{\ker(\Gamma_1)}
\end{equation}
that are restrictions of $A$.
It is proven in \cite[Sec.\,5]{Ryzh_spec} that $\widetilde{A}$ is
symmetric with equal deficiency indices. Moreover,
\cite[Prop.\,5.1]{Ryzh_spec} asserts that
 $\widetilde{A}$ does not depend on the parameter operator $\Lambda,$ contrary to what could be surmised from (\ref{eq:B-definition}).

\subsection{Resolvent formulae for general boundary-value problems}

\label{A_albet}

Still following \cite{Ryzh_spec}, we let $\alpha$ and $\beta$ be linear operators in the Hilbert
space $\mathcal{E}$ such that $\dom(\alpha)\supset\dom(\Lambda)$
and $\beta$ is bounded on $\mathcal{E}$.
Additionally, assume that
$\alpha+\beta\Lambda$ is closable and denote $\ss:=\overline{\alpha+\beta\Lambda}$. Consider
the linear set
\begin{equation}
  \label{eq:extension-of-boundary-conditions}
 {\mathcal H}_{\text \ss}:=\left\{A_{0}^{-1}f\dotplus\Pi\phi:f\in\mathcal{H},\,\,\phi
    \in\dom(\ss)\right\}
\end{equation}
Following \cite[Lem.\,4.1]{Ryzh_spec}, the identity
\[
(\alpha\Gamma_{0}+\beta\Gamma_{1})(A_0^{-1}f+\Pi\phi)=\beta\Pi^*f+(\alpha+\beta\Lambda)\phi,\qquad f\in {\mathcal H},\ \phi\in\dom(\Lambda),
\]
implies that $\alpha\Gamma_{0}+\beta\Gamma_{1}$ is correctly defined on
$\dom(A_{0})\dotplus\Pi\dom(\Lambda).$
The assumption that $\alpha+\beta\Lambda$ is closable is used to
extend the domain of definition of $\alpha\Gamma_{0}+\beta\Gamma_{1}$ to the set
\eqref{eq:extension-of-boundary-conditions}. Moreover, one shows that
${\mathcal H}_{\text \ss}$
is a Hilbert space with respect to the norm
\[
\Vert u\Vert_{\ss}^2:=\Vert f\Vert_{\mathcal H}^2+\Vert\phi\Vert_{\mathcal E}^2+\Vert\ss\phi\Vert_{\mathcal E}^2,\quad u=A_0^{-1}f+\Pi\phi.
\]
It follows that the constructed extension $\alpha\Gamma_0+\beta\Gamma_1$ is a bounded operator from ${\mathcal H}_{\ss}$ to $\mathcal E.$

According to
\cite[Thm.\,4.1]{Ryzh_spec}, if the operator
$\overline{\alpha+\beta M(z)}$ is boundedly invertible for
$z\in\rho(A_{0})$, then, on the one hand, the spectral BVP
\begin{equation*}
    \begin{cases}
    (A- zI)u=f,\\[0.3em]
    \left(\alpha\Gamma_0+\beta\Gamma_{1}\right) u
    = \phi,\qquad f\in\mathcal{H},\,\,\phi\in\mathcal{E},
  \end{cases}
\end{equation*}
has a unique solution $u\in{\mathcal H}_{\ss},$ where, as above, $\alpha\Gamma_0+\beta\Gamma_{1}$ is a bounded operator on
${\mathcal H}_{\ss}.$
On the other hand, it follows from
\cite[Thm.\,5.1]{Ryzh_spec} that the function
\begin{equation}
  (A_{0}-z I)^{-1}-(I-zA_{0}^{-1})^{-1}\Pi
  [\overline{\alpha+\beta M(z)}]^{-1}\beta\Pi^{*}
(I-zA_{0}^{-1})^{-1}
\label{Krein_formula1}
\end{equation}
is the resolvent of a closed operator\label{Aab} $A_{\alpha\beta}$ densely
defined in $\mathcal{H}$. Moreover,
$\widetilde{A}\subset A_{\alpha\beta}\subset A$ and
$\dom(A_{\alpha\beta})\subset\{u\in{\mathcal H}_{\ss}:
(\alpha\Gamma_{0}+\beta\Gamma_{1})u=0\}$.

\section{Large-coupling asymptotics for a transmission problem}

\label{sec:example}

The aim of this section is to translate a problem familiar to the application-minded reader into the language of boundary triple theory presented above and obtain new results for this problem.


\subsection{Problem formulation}

Suppose that $\Omega$ is a bounded $C^{1,1}$ domain, and $\Gamma\subset\Omega$ is a closed $C^{1,1}$ curve, so that $\Gamma=\partial\Omega_-$ is the common boundary of domains $\Omega_+$ and $\Omega_-,$ where $\Omega_-$ is strictly contained in $\Omega,$ such that $\overline{\Omega}_+\cup\overline{\Omega}_-=\overline{\Omega},$ see Fig.\,\ref{fig:kaplya}.

For $a>0,$ $z\in{\mathbb C}$
we consider the ``transmission" eigenvalue problem ({\it cf.} \cite{Schechter})
\begin{equation}
\label{eq:transmissionBVP}
\begin{cases}
&-\Delta u_+=zu_+\ \ {\rm in\ } \Omega_+,\\[0.4em]
&-a\Delta u_-=zu_-\ \ {\rm in\ } \Omega_-,\\[0.4em]
&u_+=u_-,\quad \dfrac{\partial u_+}{\partial n_+}+a\dfrac{\partial u_-}{\partial n_-}=0\ \ {\rm on\ } \Gamma,\\[0.7em]
&
\dfrac{\partial u_+}{\partial n_+}=0\ \
{\rm \ on\ } \partial\Omega,
\end{cases}
\end{equation}
where $n_\pm$ denotes the exterior normal (defined a.e.) to the corresponding part of the boundary.{\footnote{The Neumann boundary condition on $\partial\Omega$ can be replaced by a Robin condition with an arbitrary coupling
constant without affecting the analysis of this section.}} The above problem is understood in the strong sense,
{\it i.e.} $u_\pm\in H^2(\Omega_\pm),$ the Laplacian differential expression $\Delta$ is the corresponding combination of second-order weak derivatives, and the boundary values of $u_\pm$ and their normal derivatives are understood in the sense of traces according to the embeddings of $H^2(\Omega_\pm)$ into $H^s(\Gamma),$ $H^s(\partial\Omega),$ where $s=3/2$ or $s=1/2.$

\subsection{A reformulation in the language of triples}

In order to make our framework applicable to (\ref{eq:transmissionBVP}), we first consider its weak formulation. We then apply the regularity theory for elliptic BVPs, see
{\it e.g.} \cite{Schechter}, to show that its solutions are in fact the solutions to (\ref{eq:transmissionBVP}). Indeed, the results of \cite{Ryzh_spec}, see also (\ref{Krein_formula1}), show that the problem (\ref{eq:transmissionBVP}) in the weak formulation is the eigenvalue problem for a self-adjoint operator of an appropriate class $A_{\alpha\beta}$ introduced in Section \ref{A_albet}, and thus its solutions are in the domain of this operator. Then the result of \cite{Schechter} is used to show that these solutions have higher regularity, as required for the solvability of (\ref{eq:transmissionBVP}) in the strong sense.

We define the Dirichlet-to-Neumann map\footnote{For convenience, we define the
  Dirichlet-to-Neumann map via $-\partial u/\partial n\vert_{\partial\Omega}$ instead of the more common $\partial u/\partial n\vert_{\partial\Omega}$. As a side note, we mention that this is obviously not the only possible choice for the operator $\Lambda.$ In particular, the trivial option $\Lambda=0$ is always possible. Our choice of $\Lambda$ is motivated by our interest in the analysis of classical boundary conditions.}
\[
\Lambda:\phi\mapsto-\dfrac{\partial u_+}{\partial n_+}-a\dfrac{\partial u_-}{\partial n_-},\qquad \phi\in H^1(\Gamma),
\]
where $u_\pm$  are the harmonic functions in $\Omega_\pm$ subject to the above boundary condition on $\partial\Omega$ and the condition $u=\phi$ on $\Gamma.$
Clearly $\Lambda=\Lambda^++a\Lambda^-,$ where $\Lambda^+,$ $\Lambda^-$ are the Dirichlet-to-Neumann maps on each side of the interface $\Gamma,$ which are self-adjoint operators in $L^2(\Gamma),$ with domain $H^1(\Gamma).$ For sufficiently large values of $a,$ the operator $\Lambda$ has the same properties as $\Lambda^+$ and $\Lambda^-,$ see \cite[Lemma 2.1]{CherErKis}.

Translating the spectral BVP (\ref{eq:transmissionBVP}) into the language of the abstract framework developed in Section \ref{BTF_sec},
we define $A_0$ as the ``Dirichet decoupling" $A_0^+\oplus a A_0^-,$ corresponding to the decomposition $L^2(\Omega_+)\oplus L^2(\Omega_-)=L^2(\Omega)=:{\mathcal H},$ where $A_0^+$ is the Laplace operator with Dirichlet condition on $\Gamma$ and Neumann condition on $\partial\Omega,$ and $A_0^-$ is the Dirichlet Laplacian on $\Omega_-.$
 Furthermore, in the context of the transmission problem (\ref{eq:transmissionBVP}), the boundary space is given by $L^2(\Gamma)=:{\mathcal E}$ and  the abstract operator $\Pi$ of Section \ref{BTF_sec} is simply the Poisson operator of harmonic lift from ${\mathcal E}$ to ${\mathcal H}$ (subject to the Neumann condition on $\partial\Omega$), while its left
inverse is the operator of trace on $\Gamma$ for functions that are harmonic on $\Omega_-$ and $\Omega_+,$ possessing square summable boundary values on $\Gamma,$  and
$\Gamma_{0}$ is the null extension of the latter to $\bigl(H^2(\Omega_-)\cap H^1_0(\Omega_-)\bigr)\dotplus\bigl(H^2(\Omega_+)  \cap H^1_{0,\Gamma}(\Omega_+)\bigr)\dotplus\Pi
L^{2}(\Gamma),$ where $H^1_{0,\Gamma}(\Omega_+)$ consists of functions in $H^1(\Omega_+)$ with zero trace on $\Gamma,$ {\it cf.} (\ref{qform}).

The problem (\ref{eq:transmissionBVP}) is then written in the form $A_{\alpha\beta}u=zu$ with $\alpha=0, \beta=I,$
equivalently
\[
Au=zu,\qquad
\Gamma_1u=0,
\]
where $A$ is defined by (\ref{eq:definition-A}) and $\Gamma_1$ is defined by (\ref{eq:definition-G-kappa}).

Finally, the operator $M(z)$ of Definition \ref{def:m-function}  is the mapping
\[
M(z):\phi\mapsto
-\dfrac{\partial u_+}{\partial n_+}-a\dfrac{\partial u_-}{\partial n_-},\qquad \phi\in H^1(\Gamma),
\]
where $u_+,$ $u_-$ solve
\begin{equation*}
\begin{cases}
&-\Delta u_+=zu_+\ \ {\rm in\ } \Omega_+,\\[0.35em]
&-a\Delta u_-=zu_-\ \ {\rm in\ } \Omega_-,\\[0.35em]
&u_+=u_-=\phi\ \ {\rm on\ } \Gamma,\\[0.35em]
&
\dfrac{\partial u_+}{\partial n_+}=0\ \
{\rm \ on\ } \partial\Omega,
\end{cases}
\end{equation*}
and the formula (\ref{M_operator})
expresses $M(z)$ in terms of $\Lambda$ and the decoupling
$A_0^+\oplus A_0^-.$
Recall (Section \ref{M_sec}) that $M(z)$ is an (unbounded) operator-valued Herglotz function, analytic in ${\mathbb C}\setminus{\mathbb R}$ so that $M(z)=M^*(\overline{z}),$ $z\in{\mathbb C}\setminus{\mathbb R}.$  In addition, for all values $z$ outside a discrete set of points, $M(z)$ is invertible and its inverse is compact.  Similarly, Definition  \ref{def:m-function} yields the $M$-operators $M^\pm$ corresponding to the components $\Omega_\pm,$
so that
\[
M^+(z):\phi\mapsto
-\dfrac{\partial u_+}{\partial n_+},\qquad M^-(z):\phi\mapsto
-a\dfrac{\partial u_-}{\partial n_-},\qquad \phi\in H^1(\Gamma),
\]
where $u_+$ and $u_-$ are as above.

By a similar argument to the one of \cite[Theorem 3.3, Theorem 5.1]{Ryzh_spec}, the following statement holds, {\it cf.} \cite{BMNW2008}.

\begin{proposition}
The spectrum of (\ref{eq:transmissionBVP}), {\it i.e.} the set of values $z$ for which (\ref{eq:transmissionBVP}) has a nonzero solution in the strong sense as described above,
coincides with the set of $z$ to which the inverse of $\alpha+\beta M(z)\equiv M(z)$ does not admit an analytic continuation.
\end{proposition}
The representation (\ref{M_operator})
applied to $M^-(z)$
implies that
\begin{equation}
M^-(z)= a\Lambda^- + z\Pi_-^*\bigl(I-a^{-1}z (A_0^-)^{-1}\bigr)^{-1}\Pi_-
=a\Lambda^-+z\Pi_-^*\Pi_-+O(a^{-1}),
\label{eq:Mstiff_asymp}
\end{equation}
where $A_0^-$ is the Dirichlet Laplacian on $\Omega_-,$ and $\Pi_-$ is the harmonic lift from $\Gamma$ to $\Omega_-.$


\subsection{Asymptotic analysis and the main result}
In what follows we analyse the resolvent of the operator $A_a$ of the transmission problem (\ref{eq:transmissionBVP}), which coincides with the resolvent of the operator $A_{0 I},$ in terms of the notation of Section \ref{sec:triples-bvp}.  In particular, the spectrum of $A_a$ coincides with the spectrum of  (\ref{eq:transmissionBVP}). Our approach is based on the use of the Kre\u\i n formula \eqref{Krein_formula1} with $\alpha=0,$ $\beta=I,$ where for the asymptotic analysis of $M(z)^{-1}$ we employ (\ref{eq:Mstiff_asymp}) and  separate the singular and non-singular parts of $\Lambda^-.$

To this end, note first that the spectrum of $\Lambda^-$
 consists of
the values $\mu$  (``Steklov eigenvalues") such that the problem
\vskip -0.5cm
\begin{equation*}
\begin{cases}
\Delta u=0,\ \ u\in H^2(\Omega_-),&\\[0.4em]
\dfrac{\partial u}{\partial n}=-\mu u\ \ {\rm on}\ \Gamma,&
\end{cases}
\end{equation*}
has a non-trivial solution.
The least (by absolute value) Steklov eigenvalue is zero, and
 the associated normalised eigenfunction
 (``Steklov eigenvector") is  $\psi_*:=|\Gamma|^{-1/2}{\mathbbm 1}_\Gamma\in{\mathcal E}\equiv L^2(\Gamma).$
 Introduce the corresponding orthogonal  projection $P:=\langle \cdot, \psi_*\rangle_{\mathcal E}\psi_*,$ which is a spectral projection relative to $\Lambda^-,$ and
decompose the boundary space ${\mathcal E}$:
\begin{equation}
{\mathcal E}=P{\mathcal E}\oplus P^\perp{\mathcal E},
\label{space_decomp}
\end{equation}
where $P^\perp:=I-P.$
This yields the following matrix representation for
$\Lambda^-:$
$$
\Lambda^-=\begin{pmatrix}
                        0& 0 \\[0.4em]
                        0 & \Lambda^-_\bot
                      \end{pmatrix},
$$
where $\Lambda^-_\bot:=P^\perp\Lambda^-P^\perp$ is treated as a self-adjoint operator in $P^\perp{\mathcal E}$. 

We write the operator $M(z)$
as a block-operator matrix relative to the decomposition (\ref{space_decomp}),
followed by an application of the Schur-Frobenius inversion formula, see \cite[Theorem 2.3.3]{Tretter}.
To this end, notice that for all $\psi\in\dom \Lambda$ one has $P\psi\in \dom \Lambda,$ and therefore $P^\perp\Lambda P$ is well defined on $\dom \Lambda.$ Similarly, $P^\perp\psi=\psi-P\psi\in \dom \Lambda,$ and $P \Lambda P^\perp$ is also well defined. Furthermore, by the self-adjointness of $\Lambda,$ one has
$
P \Lambda P^\perp\psi
=\bigl\langle P^\perp\psi,\Lambda\psi_*\bigr\rangle \psi_*,
$
and therefore $\bigl\| P \Lambda P^\perp\bigr\|_{{\mathcal E}\to{\mathcal E}}\leq\|\Lambda\psi_*\|_{\mathcal E}.$ It follows that $P \Lambda P^\perp$ is extendable to a bounded mapping on $P^\perp{\mathcal E}.$
 A similar calculation applied to $P^\perp\Lambda P$ and $P \Lambda P$ shows that these
 are extendable to bounded mappings on $P{\mathcal E}.$
  Therefore, for each $z\in\rho(A_0^+)\cap\rho(A_0^-)$ the operator $M(z)$ admits the
representation


\vskip -0.8em
\begin{equation*}
M(z)=\begin{pmatrix}
       {\mathbb A} & {\mathbb B} \\[0.4em]
       {\mathbb E} & {\mathbb D}
     \end{pmatrix}, \qquad {\mathbb A}, {\mathbb B}, {\mathbb E} \text{ bounded.}
\end{equation*}

For evaluating $M(z)^{-1}$ we use the Schur-Frobenius inversion formula \cite{Fuerer}, \cite[Theorem 2.3.3]{Tretter}
\begin{equation}
\label{eq:SchurF}
{\begin{pmatrix}
       {\mathbb A} & {\mathbb B} \\
       {\mathbb E} & {\mathbb D}
     \end{pmatrix}}^{-1}=
\begin{pmatrix}
  {\mathbb A}^{-1}+{{\mathbb A}^{-1}{\mathbb B}}{\mathbb S}^{-1}{\mathbb E}{\mathbb A}^{-1} & -{\mathbb A}^{-1}{\mathbb B}{\mathbb S}^{-1} \\[0.4em]
  -{\mathbb S}^{-1}{\mathbb E}{\mathbb A}^{-1} & {\mathbb S}^{-1}
\end{pmatrix},\qquad
{\mathbb S}:={\mathbb D}-{\mathbb E} {\mathbb A}^{-1} {\mathbb B}.
\end{equation}


Using the fact that ${\mathbb S}^{-1}=(I-{\mathbb D}^{-1}{\mathbb E}{\mathbb A}^{-1}{\mathbb B})^{-1}{\mathbb D}^{-1},$ where $\|{\mathbb D}^{-1}\|\leq Ca^{-1},$ and therefore ${\mathbb S}$  is boundedly invertible with a uniformly small bound,
we obtain (see \cite{CherErKis} for details)
\begin{equation}
\label{eq:estimatefortheorem}
M(z)^{-1}=\begin{pmatrix}
       {\mathbb A} & {\mathbb B} \\[0.3em]
       {\mathbb E} & {\mathbb D}
     \end{pmatrix}^{-1}=
\begin{pmatrix}
  {\mathbb A}^{-1}& 0\\[0.3em]
  0 & 0
\end{pmatrix}
+O(a^{-1}).
\end{equation}

\begin{theorem}
\label{thm:NRA}
Fix $\sigma>0$ and a compact set $K\subset{\mathbb C},$ and denote $K_\sigma:=\bigl\{z\in K: {\rm dist}(z, \mathbb R)\geq \sigma\bigr\}.$
There exist $C, a_0>0$ such that for all $z\in K_\sigma,$ $a\geq a_0$ one has
$$
\bigl\|({A}_a-z)^{-1}-(A_{P^\perp,P}-z)^{-1}\bigr\|_{L^2(\Omega)\to L^2(\Omega)}\leq Ca^{-1}.
$$
\end{theorem}
\begin{proof}
We use (\ref{Krein_formula1}) with $\alpha=0,$ $\beta=I$ for the resolvent $({A}_a-z)^{-1}$
and
with $\alpha=P^\perp,$ $\beta=P$ for $\bigl(A_{P^\perp,P}-z\bigr)^{-1}.$
In the former case  we use (\ref{eq:estimatefortheorem}) and in the latter case we write
$$
\bigl(\overline{P^\perp+PM(z)}\bigr)^{-1}P=P\bigl(PM(z)P\bigr)^{-1}P,
$$
by the Schur-Frobenius inversion formula \cite[Section 1.6]{Tretter}, see \eqref{eq:SchurF}.\footnote{We remark that $P^\perp+PM(z)$ is triangular (${\mathbb A}=PM(z)P,$ ${\mathbb B}=PM(z)P^\perp,$ ${\mathbb E}=0,$ ${\mathbb D}=I$ in \eqref{eq:SchurF}) with respect to the decomposition $\mathcal E=P\mathcal E\oplus P^\perp\mathcal E$.}
The claim follows by comparing the obtained expressions for the two resolvents.
\end{proof}




We now rewrite the result of Theorem \ref{thm:NRA} in a block-matrix form relative to the decomposition
$\mathcal{H}=P_-\mathcal{H}\oplus P_+\mathcal{H}=L^2(\Omega_-)\oplus L^2(\Omega_+),$ where $P_-,P_+$ are orthogonal projections from $L^2(\Omega)$ to $L^2(\Omega_-),L^2(\Omega_+),$ respectively.
This allows us to express the asymptotics of $(A_a-z)^{-1}$ in terms of the generalised resolvent \cite{Naimark1940,Naimark1943}
$R_a(z):=P_+ (A_a-z)^{-1}P_+,$ analysed next.

\begin{proposition}
One has
\begin{equation*}
R_a(z)=(A_0^+-z)^{-1}-
 \gamma_z^+\bigl(M^+(z)+M^-(z)\bigr)^{-1}(\gamma_{\bar z}^+)^*,
\end{equation*}
where $\gamma_z^+: \phi\mapsto u$ is the solution operator of the BVP
\begin{equation}
\begin{cases}
-\Delta u-z u=0,\quad u\in \dom A_0^+\dotplus \ran\Pi_+,\\[0.2em]
u|_\Gamma= \phi,\\[0.3em]
\dfrac{\partial u}{\partial n_+}=0\ \ {\rm on\ }\partial\Omega.
\end{cases}
\label{gamma}
\end{equation}
\end{proposition}
\begin{proof}
By the definition of $M^+,$ $M^-$ and a direct application of (\ref{Krein_formula1}), as in \cite[Lemma 3.2]{CherErKis}.
\end{proof}

A comparison of the latter result with  \eqref{Krein_formula1} taking into account \eqref{eq:solution-operator} yields:




\begin{corollary}
The generalised resolvent $R_a(z)$ is the solution operator for the BVP
\begin{equation*}
\begin{cases}
-\Delta u-zu=f, \quad f\in L^2(\Omega_+),
\\[0.4em]
\Gamma_1^+ u=-M^-(z)\Gamma_0^+u,
\\[0.4em]
\dfrac{\partial u}{\partial n_+}=0
\ \ {\rm on\ }\partial\Omega.\nonumber
\end{cases}
\end{equation*}
\end{corollary}
Theorem \ref{thm:NRA} now implies an operator-norm asymptotics for the generalised resolvents $R_a$ as $a\to\infty$.
\begin{theorem}\label{thm:NRA_gen}
For all $z\in K_\sigma$ the operator $R_a(z)$
admits the asymptotics
$R_a(z)=R_{\rm eff}(z)+O(a^{-1}),$ as $a\to\infty,$ in the operator-norm topology,
where $R_{\rm eff}(z)$ is the solution operator for the BVP
\begin{equation}
\label{eq:BVPBz}
\begin{cases}
-\Delta u-zu=f, \quad f\in L^2(\Omega_+),\\[0.35em]
\alpha(z)\Gamma_0^+u+ \beta\Gamma_1^+u=0,\\[0.35em]
\dfrac{\partial u}{\partial n_+}=0 \text{ on } \partial \Omega,
\end{cases}
\end{equation}
with $\alpha(z)=P^\perp+PM^-(z)P$ and $\beta=P$.


\end{theorem}

\begin{proof}
On the one hand, by Theorem \ref{thm:NRA}, the resolvent $({A}_a-z)^{-1}$ is $O(a^{-1})$-close to
$$
\bigl(A_{P^\perp,P}-z\bigr)^{-1}=(A_0-z)^{-1}-\gamma_z\bigl(\overline{P^\perp+P M(z)}\bigr)^{-1}P \gamma_{\bar z}^*,
$$
and therefore
\begin{equation}
\begin{aligned}
R_a(z)&=P_+ (A_0-z)^{-1} P_+- P_+ \gamma_z\bigl(\overline{P^\perp+P M(z)}\bigr)^{-1}P \gamma_{\bar z}^*P_+
+O(a^{-1})
\\[0.4em]
&=(A_0^+-z)^{-1}-\gamma_z^+\bigl(\overline{P^\perp +P M(z)}\bigr)^{-1}P(\gamma_{\bar z}^+)^*+O(a^{-1})
\\[0.4em]
&
=(A_0^+-z)^{-1}-\gamma_z^+ P\bigl(PM(z)P\bigr)^{-1}P (\gamma_{\bar z}^+)^*+O(a^{-1}).
\end{aligned}
\label{eq:calc_gen_res}
\end{equation}

On the other hand, by (\ref{Krein_formula1}) and applying the inversion formula \eqref{eq:SchurF}, we obtain
\begin{equation}
\begin{aligned}
R_{\rm eff}(z)&=\bigl(A_0^+-z\bigr)^{-1}-\gamma_z^+\bigl(\overline{P^\perp+PM^-(z)P+PM^+(z)}\bigr)^{-1}P(\gamma_{\bar z}^+)^*
\\[0.3em]
&=\bigl(A_0^+-z\bigr)^{-1}-\gamma_z^+P\bigl(PM^+(z)P+PM^-(z)P\bigr)^{-1}P(\gamma_{\bar z}^+)^*.
\end{aligned}
\label{eq:BVPBz1}
\end{equation}
Comparing the right-hand sides of \eqref{eq:calc_gen_res} and \eqref{eq:BVPBz1} completes the proof.
\end{proof}

Theorem \ref{thm:NRA_gen} can be further clarified by considering
 the ``truncated'' boundary space\footnote{In what follows we consistently supply the (finite-dimensional)  ``truncated'' spaces and operators pertaining to them by the breve overscript.}
$\breve {\mathcal E}:=P\mathcal E.$
Introduce the truncated harmonic lift
by $\breve{\Pi}_+:=\Pi_+|_{\breve{\mathcal E}}$ and Dirichlet-to-Neumann map $\breve{\Lambda}^+:=P\Lambda^+\vert_{\breve {\mathcal E}}.$

\begin{theorem}
\label{thm:NRA_gen_truncated}
Denote $\breve{M}^+(z):=PM^+(z)\vert_{\breve{\mathcal E}}=\breve{\Lambda}^+ + z \breve{\Pi}^*_+\bigl(1-z (A_0^+)^{-1}\bigr)^{-1}\breve{\Pi}_+.$
The formula
\begin{equation}
\label{eq:gen_res_trunc}
R_{\rm eff}(z)=\bigl(A_0^+-z\bigr)^{-1}-\breve{\gamma}_z^+\bigl(\breve{M}^+(z)+PM^-(z)P\bigr)^{-1} (\breve{\gamma}_{\bar z}^+)^*
\end{equation}
holds, where $\breve{\gamma}_z^+:\phi\mapsto u_\phi$ is the solution operator of the problem
\begin{equation*}
\begin{cases}
-\Delta u_\phi-z u_\phi=0,\ \ u_\phi\in \dom A_0^+\dotplus \ran\breve{\Pi}_+,\\[0.1cm]
\Gamma_0^+ u_\phi= \phi, \quad \phi\in \breve{\mathcal E}.
\end{cases}
\end{equation*}
\end{theorem}

Note that here $PM^-(z)P$, as opposed to $\breve{M}^+(z)$, is attributed the meaning of an operator defining nonlocal boundary conditions in the BVP.

\begin{proof}
By the definition of $\gamma_z^+$, one has $\gamma_z^+=(I-z(A_0^+)^{-1})^{-1}\Pi_+$, and therefore $\breve{\gamma}_z^+=\gamma_z^+|_{\breve{\mathcal E}}$. It follows that \eqref{eq:BVPBz1} is equivalent to (\ref{eq:gen_res_trunc}), whence the assertion of Theorem follows.
\end{proof}


\begin{corollary}
The operator $R_{\rm eff}(z)$ is the solution operator  of the problem
\begin{equation*}
\begin{cases}
-\Delta u-zu =f, \quad f\in L^2(\Omega_+),\quad u\in \dom A_0^+\dotplus \ran \breve{\Pi}_+,\\[0.4em]
\,P\Gamma_1^+u
=-PM^-(z)P \Gamma_0^+ u.
\end{cases}
\end{equation*}
\end{corollary}

Equipped with Theorems \ref{thm:NRA_gen} and \ref{thm:NRA_gen_truncated}, we provide a more convenient representation for the asymptotics of $({A}_a-z)^{-1}$ obtained in Theorem \ref{thm:NRA}.

\begin{theorem}
Denote $\mathfrak{K}_z:=\Gamma_0^+|_{\ran(\gamma_z^+P)}$,
$z\in \mathbb C_\pm$. For the resolvent $({A}_a-z)^{-1}$ one has
$$
({A}_a-z)^{-1}=\begin{pmatrix}
R_{\rm eff}(z)&\ \ \bigl(\mathfrak K_{\bar z}\bigl[R_{\rm eff}(\bar z)-(A_0^{+}-\bar z)^{-1}\bigr]\bigr)^*\Pi_-^*\\[0.8em] \Pi_-\mathfrak{K}_z \bigl[R_{\rm eff}(z)-(A_0^{+}-z)^{-1}\bigr] & \ \ \Pi_-\mathfrak K_{z}\bigl(\mathfrak K_{\bar z}\bigl[R_{\rm eff}(\bar z)-(A_0^{+}-\bar z)^{-1}\bigr]\bigr)^*\Pi_-^*
\end{pmatrix}
+O(a^{-1}),
$$
where the operator matrix is written with respect to the decomposition
$L^2(\Omega)=L^2(\Omega_+)\oplus L^2(\Omega_-).$
\end{theorem}

\begin{proof}
First, we note that since $\ran(\gamma_z^+P)$ is one-dimensional, the operator $\mathfrak{K}_z$ is well defined as a bounded linear operator from $\ran(\gamma_z^+P)$ to $\mathcal E,$ where the former is equipped with the standard norm of $L^2(\Omega_+)$ . We proceed by representing the operator $(A_{P^\perp ,P}-z)^{-1},$ see Theorem \ref{thm:NRA}, in a block-operator matrix form relative to the orthogonal decomposition ${\mathcal H}=L^2(\Omega_+)\oplus L^2(\Omega_-)$. The upper left matrix entry  $P_+ (A_{P^\perp ,P}-z)^{-1} P_+$ due to Theorem  \ref{thm:NRA_gen} is $O(a^{-1})$-close to $R_{\rm eff}(z).$ Next we consider the lower left matrix entry:
\begin{align*}
P_-\bigl(A_{P^\perp , P} -z\bigr)^{-1}P_+&= -\gamma^-_zP\bigl (PM^+(z)P+PM^-(z)P\bigr)^{-1}P\bigl(\gamma^+_{\bar z}\bigr)^*
\\[0.3em]
&
=-\gamma^-_z\Gamma_0^+ \gamma^+_zP\bigl (PM^+(z)P+PM^-(z)P\bigr)^{-1}P\bigl(\gamma^+_{\bar z}\bigr)^*\\[0.4em]
&=\gamma^-_z \Gamma_0^+\bigl[R_{\rm eff}(z)-(A_0^{+}-z)^{-1}\bigr]
=\gamma^-_z \mathfrak{K}_z\bigl[R_{\rm eff}(z)-(A_0^{+}-z)^{-1}\bigr],
\end{align*}
where $\gamma_z^-: \phi\mapsto u_\phi$ is the solution operator of the BVP ({\it cf.} (\ref{gamma}))
\begin{equation*}
\begin{cases}
-\Delta u_\phi-z u_\phi=0,\quad u_\phi\in \dom A_0^-\dotplus \ran\Pi_-,\\[0.2em]
\Gamma_0^-u_\phi= \phi.
\end{cases}
\end{equation*}
Here in the second equality we use the fact that $\Gamma_0^+ \gamma^+_z=I,$ and in the third equality we use (\ref{Krein_formula1}),
see also (\ref{eq:BVPBz1}).
Passing over to the top-right entry,
we write
\begin{align*}
P_+\bigl(A_{P^\perp ,P} -z\bigr)^{-1}P_-&= -\gamma^+_zP\bigl (P M^+(z)P+PM^-(z)P\bigr)^{-1}P(\gamma^-_{\bar z})^*
\\[0.3em]
&
=\bigl(\mathfrak K_{\bar z}\bigl[R_{\rm eff}(\bar z)-(A_0^{+}-\bar z)^{-1}\bigl]\bigr)^*(\gamma^-_{\bar z})^*
\\[0.25em]
&
=\bigl(\mathfrak K_{\bar z}\bigl[R_{\rm eff}(\bar z)-(A_0^{+}-\bar z)^{-1}\bigr]\bigr)^*\Pi_-^*\bigl(1-z (A_0^{-})^{-1}\bigr)^{-1},
\end{align*}
and the claim pertaining to the named entry follows by a virtually unchanged argument. Finally, for the bottom-right entry
we have
$$
P_-\bigl(A_{P^\perp , P} -z\bigr)^{-1}P_- =
(A_0^{-}-z)^{-1}+\gamma^-_z \mathfrak K_{z}\bigl(\mathfrak K_{\bar z}\bigl[R_{\rm eff}(\bar z)-\bigl(A_0^{+}-\bar z\bigr)^{-1}\bigr]\bigr)^*\bigl(\gamma^-_{\bar z}\bigr)^*,
$$
which completes the proof.
\end{proof}

The representation for $R_{\rm eff}(z)$ given by Theorem \ref{thm:NRA_gen_truncated}
allows us to further simplify the asymptotics of
$({A}_a-z)^{-1},$ using the fact that
$$
PM^-(z)P=P\Lambda^-P+zP\Pi_-^*\Pi_-P + O(a^{-1})=z\breve{\Pi}_-^*\breve{\Pi}_- +O(a^{-1}),\quad \breve{\Pi}_-:=\Pi_-\vert_{\breve{\mathcal E}}.
$$
As a result, one has\footnote{We remark that the generalised resolvent $\widetilde{R}_{\rm eff}(z)$ is thus the solution operator of a spectral BVP of a special class. Namely, the dependence  of its boundary condition on the spectral parameter is linear with respect to the latter. Such spectral BVPs were considered e.g. in \cite{Shkalikov_1983}.}
\begin{equation*}
R_{\rm eff}(z)=\widetilde{R}_{\rm eff}(z)+O(a^{-1}),\quad\ \
\widetilde{R}_{\rm eff}(z):= (A_0^+-z)^{-1}-\breve{\gamma}_z^+\bigl(\breve{M}^+(z)+z\breve{\Pi}_-^*\breve{\Pi}_-\bigr)^{-1} (\breve{\gamma}_{\bar z}^+)^*,
\end{equation*}
and hence the following result holds.
\begin{theorem}
\label{thm:main_fin}
The resolvent $({A}_a-z)^{-1}$ has the following asymptotics in the operator-norm topology:
\begin{equation}
({A}_a-z)^{-1}=\begin{pmatrix}
\widetilde{R}_{\rm eff}(z)&\ \ \bigl(\mathfrak K_{\bar z}\bigl[\widetilde{R}_{\rm eff}(\bar z)-(A_0^{+}-\bar z)^{-1}\bigr]\bigr)^*\breve{\Pi}_-^*\\[0.8em] \breve{\Pi}_-\mathfrak{K}_z \bigl[\widetilde{R}_{\rm eff}(z)-(A_0^{+}-z)^{-1}\bigr] & \ \ \breve{\Pi}_-\mathfrak K_{z}\bigl(\mathfrak K_{\bar z}\bigl[\widetilde{R}_{\rm eff}(\bar z)-(A_0^{+}-\bar z)^{-1}\bigr]\bigr)^*\breve{\Pi}_-^*
\end{pmatrix}+
O(a^{-1}),
\label{716}
\end{equation}
where the operator matrix is written with respect to the decomposition
$L^2(\Omega)=L^2(\Omega_+)\oplus L^2(\Omega_-).$
\end{theorem}

\subsection{An out-of-space extension and the ``electrostatic" problem}

Consider the space $H_{\rm eff}=L^2(\Omega_+)\oplus \mathbb C$ and the following linear subset of $L^2(\Omega):$
\begin{equation}
\label{eq:domain_fin_modI}
\dom \mathcal A_{\rm eff}=\biggl\{
\binom{u_+}{\eta}\in H_{\rm eff}:\ u_+\in H^2(\Omega_+),\ \ u_+|_\Gamma=\frac{\eta}{\sqrt{|\Omega_-|}}{\mathbbm 1}_\Gamma, \ \ \dfrac{\partial u_+}{\partial n_+}\biggr\vert_{\partial\Omega}=0\biggr\},
\end{equation}
where $u|_\Gamma$
is the trace of the function $u$  and ${\mathbbm 1}_\Gamma$ is the unity function on $\Gamma.$
On $\dom \mathcal A_{\rm eff}$ we set the action of the operator ${\mathcal A}_{\rm eff}$ by the formula
\vskip -0.3cm
\begin{equation*}
\mathcal A_{\rm eff}\binom{u_+}{\eta}=
\left(\begin{array}{c}-\Delta u_+\\[0.6em]
\dfrac{1}{\sqrt{|\Omega_-|}}\mathop{\mathlarger{\mathlarger{\int}}}_{\!\!\!\Gamma}\dfrac{\partial u_+}{\partial n_+}
\end{array}\right).
\end{equation*}

\begin{remark}
Note that the resolvent of $\mathcal A_{\rm eff}$ is the so-called  Strauss dilation \cite{Strauss,Strauss_survey} of the generalised resolvent $\widetilde{R}_{\rm eff}(z)$.
\end{remark}

Next we use \cite[Theorem 4.4]{CherErKis}, which in view of Theorem 4.6 implies the following theorem.

\begin{theorem}\label{mmm}
  The resolvent $(\mathcal A_{\rm eff}-z)^{-1}$ is unitary equivalent to the block operator matrix on the right hand side of   (\ref{716}), , treated as an operator in $L^2(\Omega_+)\oplus \Pi_-\breve{\mathcal E}$.
\end{theorem}

We remark that as it is easily seen the block operator matrix on the right hand side of  (\ref{716}) is equal to zero in the orthogonal complement to the subspace $L^2(\Omega_+)\oplus \Pi_-\breve{\mathcal E}$.



Together with the spectral theorem for self-adjoint operators, the result obtained in Theorem  \ref{mmm} immediately yields  (see, e.g., \cite{MR1192782}) the following corollary.

\begin{corollary}
The spectra of the operators $A_a$ converge in the sense of Hausdorff, as $a\to\infty,$ with an order $O(a^{-1})$ error estimate, uniformly on compact subsets of ${\mathbb C},$ to the spectrum of the operator ${\mathcal A}_{\rm eff}.$
\end{corollary}
An explicit representation for the spectrum of ${\mathcal A}_{\rm eff},$ {\it i.e.} the set of $z\in{\mathbb C}$ for which the problem
\begin{equation}
\mathcal A_{\rm eff}\binom{u_+}{\eta}=z\binom{u_+}{\eta},
\label{eff_eig}
\end{equation}
has a nontrivial solution $(u_+, \eta)^\top,$ can be obtained as follows. We represent $u_+\in H^2(\Omega_+)$ in the form $u_+=v+c,$ where $c\in{\mathbb C}$ is related to $\eta$ by the formula $c\sqrt{|\Omega_-|}=\eta,$ {\it cf.} (\ref{eq:domain_fin_modI}), and $v$ solves the problem $-\Delta v=z(v+c)$ subject to the Neumann boundary condition $({\partial v}/{\partial n_+})\vert_{\partial\Omega}=0$ and the Dirichlet boundary condition $v|_\Gamma=0,$
or equivalently $v=zc(A_0^+-z)^{-1}{\mathbbm 1}_{\Omega_+},$ where ${\mathbbm 1}_{\Omega_+}$ is the unity function on $\Omega_+.$
Therefore, in terms of the pair $(v, c)^\top$ the eigenvalue problem (\ref{eff_eig}) admits the form:
\begin{equation}
\left(\begin{array}{c}-\Delta v\\[0.4em]
\dfrac{1}{\sqrt{|\Omega_-|}}\mathop{\mathlarger{\mathlarger{\int}}}_{\!\!\!\Gamma}\dfrac{\partial v}{\partial n_+}
\end{array}\right)=z\binom{v+c}{c\sqrt{|\Omega_-|}},
\label{eig_reform}
\end{equation}
and so its solvability is easily seen to be equivalent to the boundary equation appearing in the second component of (\ref{eig_reform}):
\vskip -0.5cm
\begin{equation}
\dfrac{1}{\sqrt{|\Omega_-|}}\mathop{\mathlarger{\mathlarger{\int}}}_{\!\!\!\Gamma}\dfrac{\partial v}{\partial n_+}=zc\sqrt{|\Omega_-|}.
\label{dispersion_explicit}
\end{equation}
Suppose first that $c\neq 0.$

Denoting as in Introduction above by $\lambda_j^+,$ $j=1,2,\dots,$ and $\phi_j^+,$ $j=1,2,\dots$ the eigenvalues and the corresponding normalised eigenfunctions of the operator $A_0^+,$ the relation (\ref{dispersion_explicit}) is reduced to $c=-|\Omega|^{-1}\int_{\Omega_+} v_+$ using the Green's formula and the equality  $-\Delta v = z(v+c)$. Therefore, based on the calculation contained in Introduction, the  solvability of \eqref{eff_eig} turns out to be equivalent to  (\ref{electrostatic_spectrum}).
Alternatively, if $c=0,$ which corresponds to the case when $\eta=0,$ the function $u_+$ is clearly an eigenfunction of $A_0^+$,  and it can easily be shown to have zero mean\footnote{This situation has been shown \cite{Albert} to be non-generic, at least for the case of simply connected domains and the Dirichlet condition on the boundary.} over $\Omega_+.$

\begin{remark}
It has been conjectured \cite{Smilansky} (and established rigorously in the case of the Dirichlet Laplace operator in two dimensions \cite{EP}) that the eigenfunctions of a BVP on a sufficiently regular bounded domain $\Omega$ can be ``nearly extended" (see \cite{EP} for rigorous details) to generalised eigenfunctions of the whole-space problem, corresponding to the same eigenvalue. This can be interpreted as an effect of ``transparency'' of the domain $\Omega$ to the waves of certain wavelengths.
The set of values described by (\ref{electrostatic_spectrum}) can therefore be interpreted as the set of wavelengths at which a similar effect of transparency occurs for problems (\ref{eq:transmissionBVP}), as $a\to\infty,$ in other words for problems with low-index dielectric inclusions. A rigorous proof of this observation requires the development of scattering theory for high-contrast transparent obstacles, and it will be addressed in a separate publication.
\end{remark}

\begin{remark} The spectrum of $A^+_0,$ contrary to what would seem from Theorem \ref{thm:main_fin}, in the generic case does not enter the spectrum of (\ref{eff_eig}).
The mechanism for this is described in \cite{Mikhailova_Pavlov}.
\end{remark}

\section*{Acknowledgements}

The authors gratefully acknowledge a number of useful remarks made by the referees.

KDC is grateful for the financial support of
EPSRC: grants EP/L018802/2 and EP/V013025/1.
AVK has been supported by the RSF grant 20-11-20032.
The work of KDC and LOS was also partially supported by the grant of CONACyT CF-2019 No.\,304005.
LOS is grateful
for the financial support of PASPA-DGAPA-UNAM during his sabbatical
leave and thanks the University of Bath for their hospitality. KDC and LOS are grateful for the financial support of the Royal Society Newton Fund: Grant ``Homogenisation of degenerate equations and scattering for new materials"


\begin{thebibliography}{9}



\bibitem{Albert} J. Albert, Genericity of simple eigenvalues for elliptic PDE's, {\it Proceedings of the American Mathematical Society} 48(2): 413--418, 1975.

\bibitem{AKKL}
H. Ammari, H. Kang, K. Kim, H. Lee, 2013. Strong convergence of the solutions of the linear elasticity and uniformity of asymptotic expansions in the presence of small inclusions. {\it J. Differential Equations,} {\bf 254}, 4446--4464.










\bibitem{MR0080271}
M.~\v{S}. Birman,
\newblock On the theory of self-adjoint extensions of positive definite
  operators.
\newblock {\em Mat. Sb. N.S.,} 38(80): 431--450, 1956.

\bibitem{MR1192782}
M.~\v{S}.~Birman and M.~Z.~Solomjak.
\newblock {\em Spectral theory of selfadjoint operators in {H}ilbert space}.
\newblock Mathematics and its Applications (Soviet Series). D. Reidel Publishing Co., Dordrecht, 1987.


\bibitem{BMNW2008} M. Brown, M. Marletta, S. Naboko, and I. Wood,
  Boundary triples and {$M$}-functions for non-selfadjoint
  operators, with applications to elliptic {PDE}s and block operator
  matrices. {\it J. Lond. Math. Soc. (2),} 77(3): 700--718, 2008.



\bibitem{KCher}
K. D. Cherednichenko and A. V. Kiselev, Norm-resolvent convergence of
one-dimensional high-contrast periodic problems to a Kronig-Penney
dipole-type model. {\it Comm. Math. Phys.,} 349(2): 441--480, 2017.

\bibitem{KCherYulia}
K. D. Cherednichenko, Yu. Yu. Ershova, and A. V. Kiselev, Time-dispersive behaviour as a feature of critical contrast media,
{\it SIAM J. Appl. Math.} 79(2): 690--715, 2019.

\bibitem{CherErKis}
K. D. Cherednichenko, Yu. Ershova, and A. V. Kiselev,
\newblock Effective behaviour of critical-contrast PDEs: micro-resonances, frequency conversion, and time dispersive properties. I. {\it Communications in Mathematical Physics} 375: 1833--1884, 2020.

\bibitem{KCherYuliaNab}
K.~Cherednichenko, Y. Ershova, A.~Kiselev, and S.~Naboko,
Unified approach to critical-contrast homogenisation with explicit
links to time-dispersive media, {\it Trans. Moscow Math. Soc.} 80(2): 295--342, 2019.

\bibitem{ChKS_OTAA}
\newblock K.~Cherednichenko, A.~Kiselev and L.~Silva,
\newblock Scattering theory for non-selfadjoint extensions of symmetric operators. In: Analysis as a Tool in Mathematical Physics. {\it Oper. Theory Adv. Appl.} 276: 194--230, 2020.


\bibitem{CherednichenkoKiselevSilva}
K.~D. Cherednichenko, A.~V. Kiselev, and L.~O. Silva,
\newblock Functional model for extensions of symmetric operators and
  applications to scattering theory.
\newblock {\em Netw. Heterog. Media,} 13(2): 191--215, 2018.

\bibitem{ChKS4}
K.~D. Cherednichenko, A.~V. Kiselev, and L.~O. Silva,
\newblock Functional model for boundary-value problems. Mathematika 67 (2021), no. 3, 596--626.







\bibitem{EP} J.-P. Eckmann, C.-A. Pillet. Spectral duality for planar billiards. {\it Commun. Math. Phys.} 170: 283--313, 1995





\bibitem{Gor}
V. I. Gorbachuk and M. L. Gorbachuk. {\it Boundary value problems for operator differential equations.}
  Mathematics and its Applications (Soviet Series), 48,
  Kluwer Academic Publishers, Dordrecht, 1991.

\bibitem{HempelLienau_2000}
R. Hempel, K. Lienau, 2000. Spectral properties of the periodic media in
large coupling limit. {\it Commun. Partial Diff. Equations}  {\bf 25}, 1445--1470.









\bibitem{MR0024574}
M.~Kre\u\i n,
\newblock The theory of self-adjoint extensions of semi-bounded {H}ermitian transformations and its applications. {I}.
\newblock {\em Rec. Math. [Mat. Sbornik] N.S.,} 20(62): 431--495,  1947.

\bibitem{MR0024575}
M.~G. Kre\u\i n,
\newblock The theory of self-adjoint extensions of semi-bounded {H}ermitian
  transformations and its applications. {II}.
\newblock {\em Mat. Sbornik N.S.,} 21(63): 365--404, 1947.





 \bibitem{Mikhailova_Pavlov} A. B. Mikhailova, B. S. Pavlov, 2009. Remark on the compensation of singularities in Krein's formula. In: Methods of Spectral Analysis in Mathematical Physics, {\it Oper. Theory Adv. Appl.,} 186:325--337, Birkh\"{a}user, Basel, 2009.


\bibitem{MR573902}
S.~N. Naboko,
\newblock Functional model of perturbation theory and its applications to scattering theory.
\newblock In: Boundary Value Problems of Mathematical Physics 10.
\newblock {\em Trudy Mat. Inst. Steklov,} 147:86--114, 203, 1980.




\bibitem{Naimark1940}
M. Neumark, Spectral functions of a symmetric operator. (Russian)
{\it Bull. Acad. Sci. URSS. Ser. Math. [Izvestia Akad. Nauk SSSR],} 4: 277--318, 1940.

\bibitem{Naimark1943}
M. Neumark, Positive definite operator functions on a commutative
group. (Russian) {\it Bull. Acad. Sci. URSS Ser. Math. [Izvestia
 Akad. Nauk SSSR],} 7: 237--244, 1943.
 
\bibitem{Drogobych}
B.~S.~Pavlov.
\newblock Dilation theory and the spectral analysis of non-selfadjoint differential operators.
\newblock {\em Proc. 7th Winter School, Drogobych, 1974}, TsEMI, Moscow, 3--49, 1976. English translation: {\it Transl., II Ser., Am. Math. Soc.,} 115: 103--142, 1981.

\bibitem{Panasenko} Panasenko, G. P.
Asymptotics of the solutions and eigenvalues of elliptic equations with strongly varying coefficients. (English; Russian original)
{\it Sov. Math., Dokl.} 21: 942-947, (1980); translation from Dokl. Akad. Nauk SSSR 252: 1320-1325, 1980.












\bibitem{Ryzh_spec} V. Ryzhov, Spectral boundary value problems and
  their linear operators. In: Analysis as a Tool in Mathematical Physics. {\it Oper. Theory: Adv. Appl.} 276: 576--626, 2020.



\bibitem{Schechter}
M. Schechter, A generalization of the problem of transmission. {\it
  Ann. Scuola Norm. Sup. Pisa,} 14(3): 207--236, 1960.

\bibitem{Fuerer} Schur, I., 1905. Neue Begr\"undung der Theorie der Gruppencharaktere Sitzungsberichte der Preussischen Akademie der Wissenschaften, Physikalisch-Mathematische Klasse, 406--436.

\bibitem{Smilansky}
U. Smilansky, Semiclassical quantization of chaotic billiards. {\it In: Chaos and Quantum Chaos. W. D. Heiss, ed.} Lecture Notes in Physics 411: 57--120, Berlin: Springer, 1993

\bibitem{Strauss} A. V. \v{S}traus. Generalised resolvents of symmetric operators (Russian) \emph{Izv. Akad. Nauk SSSR, Ser. Mat.,} 18: 51--86, 1954.

\bibitem{Strauss_survey} A. V. \v{S}traus, Functional models and  generalized spectral functions of symmetric operators. {\it St. Petersburg Math. J.,} 10(5): 733-784, 1999.
 (Russian) \emph{Algebra i Analiz} 10(5): 1--76, 1998.



\bibitem{Tretter}  C. Tretter. \emph{Spectral Theory of Block Operator
  Matrices and Applications}. Imperial College Press, London, 2000.

\bibitem{MR0051404} M.~I. Vi\v{s}ik.  \newblock On general boundary
  problems for elliptic differential equations.  \newblock {\em Trudy
    Moskov. Mat. Ob\v{s}\v{c}.,} 1: 187--246, 1952.

\bibitem{Zhikov2000}
Zhikov, V. V., 2000. On an extension of the method of two-scale convergence and its applications, {\it Sbornik: Mathematics} {\bf 191}(7), 973--1014.


\bibitem{Zhikov_2004} V. Zhikov. On spectrum gaps of some divergent elliptic operators with periodic coefficients. {\it St.\,Petersburg Math. J.,} 16(5): 773--790, 2004.
    
\bibitem{Shkalikov_1983}
Shkalikov, A. A., 1983. Boundary problems for ordinary differential equations with parameter in the boundary conditions. {\it J. Soviet. Math.} {\bf 33}(6):1311-1342.


\end{thebibliography}
\end{document}